\documentclass{amsart}

\usepackage{latexsym,amsmath,amsfonts,amscd,amssymb, amsthm}

\theoremstyle{plain}
\newtheorem{theorem}{Theorem}[section]
\newtheorem{lemma}{Lemma}[section]
\newtheorem{proposition}{Proposition}[section]
\newtheorem{corollary}{Corollary}[section]
 
\theoremstyle{definition}
\newtheorem{definition}{Definition}[section]
\newtheorem{example}{Example}[section]

\theoremstyle{remark}

\title{Hermitian symmetric polynomials and CR complexity}

\author{John P. D'Angelo}

\address{Dept. of Mathematics, Univ. of Illinois, 1409 W. Green St., Urbana IL 61801}

\email{jpda@math.uiuc.edu}

\author{Ji\v{r}\'\i\ Lebl}

\address{Dept. of Mathematics, Univ. of Illinois, 1409 W. Green St., Urbana IL 61801}

\email{jlebl@math.uiuc.edu}

\begin{document}

\maketitle

\begin{abstract} Properties of Hermitian forms are used to investigate several natural 
questions from CR Geometry. To each Hermitian symmetric polynomial we assign a Hermitian form.
We study how the signature pairs of two Hermitian forms behave under the polynomial product.
We show, except for three trivial cases, that every signature pair can be obtained from
the product of two indefinite forms. We provide several new applications to the complexity theory 
of rational mappings between hyperquadrics, including a stability result about the existence of non-trivial 
rational mappings from a sphere to a hyperquadric with a given signature pair.

\medskip

\noindent
{\bf AMS Classification Numbers}: 15B57, 32V30, 32H35, 14P05.

\medskip

\end{abstract}

\section{Introduction}

This paper uses properties of Hermitian forms to investigate several natural questions from CR Geometry and to 
establish new results about both real and complex polynomials. We begin by showing why Hermitian forms arise.

Let $\rho$ be an arbitrary real-valued polynomial on ${\bf R}^{2n}$. Given a fixed complex structure
on ${\bf C}^n$, $\rho$ then determines a Hermitian form in the following fashion. Call the real variables $(x,y)$
and set $z=x+iy$ and ${\overline z} = x-iy$. 
Put $r(z, {\overline z}) = \rho({z+{\overline z} \over 2}, {z - {\overline z} \over 2i})$. Then polarize
(treat $z$ and ${\overline z}$ as independent variables) to obtain a polynomial $r(z,{\overline w})$. 
We call $r$ a Hermitian symmetric polynomial. See Section 2 for properties of such polynomials.
We write $r$ in multi-index notation as
$$ r(z, {\overline w}) = \sum c_{\alpha \beta} z^\alpha {\overline w}^\beta;\eqno (1) $$
the matrix $C=(c_{\alpha \beta})$ is Hermitian symmetric. Thus $\rho$ determines $C$ uniquely (given a fixed complex structure on ${\bf R}^{2n}$).
Furthermore, without loss of generality, one can add one variable and bihomogenize $r$. We can then regard the Hermitian form $C$
as being defined on the complex vector space $V(d,n+1)$ of homogeneous polynomials of degree $d$ in $n+1$ variables.

Conversely given such a Hermitian form on $V(d,n+1)$ we can recover $r$ and then $\rho$. We will usually
express our ideas in terms of $r$. We say that ${\bf s}(r) = (A,B)$
if $C$ has $A$ positive and $B$ negative eigenvalues, and we call $(A,B)$ the signature pair of $r$.
We write ${\bf r}(r)$ for the rank $A+B$ of $C$.
We naturally call $r$ {\it positive semi-definite} if ${\bf s}(r) = (A,0)$,
{\it negative semi-definite} if ${\bf s}(r) = (0,B)$, and {\it indefinite} if ${\bf s}(r)=(A,B)$
where both $A$ and $B$ are not zero. We warn the reader that an indefinite $r$ can be nonnegative as a function.
When we need  to know the number of vanishing eigenvalues we 
consider the {\it inertia triple}
${\bf in}(r)$. See Definition 2.2.

Our primary interest lies in the relationship
between the underlying Hermitian matrix and the zero set of $r$. 
To investigate this matter we naturally need to know what happens to signature pairs when we multiply polynomials.
Note that the sum or product
of two Hermitian symmetric polynomials is also Hermitian symmetric. What happens to the underlying
Hermitian matrices under these operations? The situation for sums is completely understood. See [Mq].  The polynomial sum and 
the matrix sum correspond precisely.
The situation for products is more difficult and new phenomena arise. In particular the Hermitian matrix corresponding
to the product of two polynomials is not easily expressed in terms of the corresponding Hermitian matrices.  In this 
paper we will
study carefully what happens to the numbers of positive and negative eigenvalues under this product.

Assume that the signature pair ${\bf s}(p_j)$ is $(A_j,B_j)$.
Let $(A,B)$ denote the signature pair for $p_1 p_2$.
The following inequalities are easy to verify:
$$ A \le A_1 A_2 + B_1 B_2 \eqno (2.1)$$
$$ B \le A_1 B_2 + A_2 B_1 \eqno (2.2) $$
$$ {\bf r}(p_1 p_2) = A+B \le (A_1 + B_1)(A_2 + B_2) = {\bf r}(p_1) {\bf r}(p_2). \eqno (2.3) $$
In each case equality holds for generic choices of
 the coefficients, but strict inequality
 is possible and plays a significant role
in our applications to CR Geometry. 

Let $(A,B)$ be an ordered pair 
of non-negative
integers. Theorem 4.1 states that there are indefinite homogeneous polynomials $r_1$ and $r_2$ with ${\bf s}(r_1 r_2) = (A,B)$ if and only if 
$(A,B)$ is not one of the three trivial cases $(0,0)$, $(1,0)$, $(0,1)$. 

In Propositions 4.1, 4.2, and in the proof of Theorem 4.1 we provide 
several elementary but striking examples. 
For example, there exists pairs of (quite special)
Hermitian symmetric polynomials with arbitrarily large ranks
whose product has signature pair $(2,0)$ and hence rank $2$. We call this phenomenon {\it collapsing of rank}. 
This kind of collapse is sharp; by Proposition 2.2 the product has  rank $1$ only if both factors do. 
A  similar result fails
for real-analytic Hermitian symmetric functions.

Theorem 4.1 completely answers a key question about signature pairs. For applications to CR Geometry, however,
we must also analyze the signature pairs of all possible multiples of a fixed polynomial.
Without loss of generality we may bihomogenize and assume that all Hermitian symmetric polynomials under consideration 
are actually bihomogeneous. We can then use these ideas
to investigate rational CR mappings between hyperquadrics. The main results may be regarded
as part of an evolving theory of {\it complexity in CR geometry}. Theorem 7.1
gives an almost  complete analysis when the domain hyperquadric is the 
sphere $S^3$.  Theorem 8.2 provides a stability result
for $S^{2n-1}$, but it does not analyze the finite set of exceptions 
arising in each dimension.

We introduce some notation.

$$ Q(a,b) = \{ z \in {\bf C}^{a+b} : \sum_{j=1}^a |z_j|^2 - \sum_{a+1}^{a+b} |z_j|^2 = 1 \} \eqno (3.1) $$

$$ HQ(A,B) = \{ Z \in {\bf P}^{A+B-1} : \sum_{j=1}^A |Z_j|^2 - \sum_{j=A+1}^{A+B} |Z_j|^2 = 0 \}. \eqno (3.2) $$
Then $HQ(a,b+1)$ is the closure in projective space of $Q(a,b)$ in affine space.

If $p(z) = 0$ on the set $\sum_{j=1}^n |z_j|^2 - |z_{n+1}|^2 = 0$, then $p$ determines a rational mapping from the unit sphere
to a hyperquadric $Q(A,B-1)$. See the discussion following (4) in Section 2.
Given the pair $(A,B)$ we naturally ask whether there are any such nontrivial
mappings, and if so, whether
there is a bound on the degree of the set of these mappings. In the special case where $B=1$, we are asking what happens when
the target is also a sphere, and our conclusions in this case clarify and unify results from many authors. 

Theorem 8.2  provides a stability result for rational mappings from 
$S^{2n-1}$ to hyperquadrics. We find an integer $N$, depending 
quadratically on  $n$,
 such that for each pair $(A,B)$ with  $A+B \ge N$ we can analyze the 
situation. 
In fact, except when the target is also a sphere, Corollary 8.2 shows that 
so-called {\it degree estimates} cannot hold when $A+B \ge N$. For a finite set
of signature pairs, the techniques 
in this paper do not answer in general whether maps exist. The general answer seems to be complicated,
and hence we are satisfied with the stability result.
Theorem 7.1 determines what happens for each pair $(A,B)$,
 with the only  exceptions being $(3,2)$ and $(2,3)$.
Analyzing the cases of small rank is difficult. We rely on
Faran's classification of planar maps [Fa2] and recent work in [L]. The table after Theorem 7.1 nicely summarizes the
situation when $n=2$. 

We discuss how other work fits into these ideas. The papers [BH] and [BEH] provide rigidity results which we can express
using the notion of signature pair. In our notation, the main result in [BH] states the following:
let $r$ denote the defining equation for $HQ(a,b)$. Assume that $a \ge b \ge 2$ and $c \ge b$. 
Let $rq$ be the Hermitian symmetric polynomial corresponding
to a rational CR mapping, expressed in lowest terms. If ${\bf s}(rq)= (c,b)$, then $q$
is a constant and $c=a$. In other words, the number of positive eigenvalues cannot increase in this setting. An example from
[BH] shows that the conclusion can fail if $a<b$. It is even possible to construct such polynomial examples 
of arbitrarily large degree when $a<b$. See [D4].

A rigidity result in [BEH] generalizes the result in [BH].
In our notation and under the same assumptions the following holds. If ${\bf s}(rq)= (c,d)$ and $b \le d < 2b-1$, then
$q$ is a constant, $c=a$, and $b=d$. Both [BH] and [BEH] also prove stronger results when the mappings are assumed to
preserve sides of the hyperquadrics.

Our work focuses on Hermitian forms, rather than on the underlying mappings. 
The results of [BH,BEH] allow the consideration of CR
mappings with linearly dependent components; these components cause 
cancellation in the corresponding Hermitian forms. See Section 5. A nonlinear mapping
from $Q(a,b)$ to $Q(c,d)$
may thus have the same Hermitian form as does a linear
embedding of $Q(a,b)$ in ${\mathbf C}^{c+d}$.

The authors acknowledge support from NSF grants DMS 07-53978 (JPD) and DMS 09-00885 (JL). They also
wish to thank AIM for the workshop on CR Complexity Theory in 2006 which helped nourish some of these ideas.

\section{Hermitian symmetric polynomials}

We begin by recalling some basic facts
about Hermitian symmetric polynomials. See [D3] for details. 

\begin{definition} Let $p:{\bf C}^n \times {\bf C}^n \to {\bf C}$ be a polynomial. We say that
$p$ is Hermitian symmetric if $p(z,{\overline w}) = {\overline {p(w,{\overline z})}}$ for all $z,w$.
\end{definition}

\begin{proposition} The following statements are equivalent:

\begin{itemize}
\item $p$ is Hermitian symmetric.

\item The function $z \to p(z,{\overline z})$ is real-valued.

\item We can write $p(z,{\overline w}) = \sum_{\alpha, \beta} c_{\alpha \beta} z^\alpha {\overline w}^\beta$ where
the matrix of coefficients is Hermitian symmetric: $c_{\alpha \beta} = {\overline {c_{\beta \alpha}}}$. 
\end{itemize} 
\end{proposition}

\begin{definition} Let $p$ be a Hermitian symmetric polynomial.
\begin{itemize}
\item The {\it rank} of $p$, written ${\bf r}(p)$, is the rank of the matrix $(c_{\alpha \beta})$.
\item The {\it signature pair} of $p$, written ${\bf s}(p)$, is $(A,B)$ if the matrix $(c_{\alpha \beta})$ has $A$ positive and $B$ negative eigenvalues.
\item $p$ is {\it positive semi-definite} if ${\bf s}(r) = (A,0)$,
{\it negative semi-definite} if ${\bf s}(r) = (0,B)$, and {\it indefinite} if ${\bf s}(r)=(A,B)$
where both $A$ and $B$ are not zero. 
\item Let $p$ be bihomogeneous of degree $(d,d)$. The {\it inertia triple} of $p$,
written ${\bf in}(r)$ is $(A,B,k)$, where ${\bf s}(p) = (A,B)$ and $k$ is the number of zero eigenvalues.
Here the matrix $C$ is regarded as a Hermitian form on the vector space of homogeneous polynomials of degree $d$.
\end{itemize}
\end{definition}

Let $p$ be Hermitian symmetric with ${\bf s}(p) = (A,B)$. 
As in [D1] we can find holomorphic polynomial mappings
$f:{\bf C}^n \to {\bf C}^A$ and $g:{\bf C}^n \to {\bf C}^B$ such that

$$ p(z,{\overline z}) = ||f(z)||^2 - ||g(z)||^2  = \sum_{j=1}^A |f_j(z)|^2 - \sum_{j=1}^B |g_j(z)|^2, \eqno (4) $$
and the components of $f$ and $g$ are linearly independent.
We write $p = ||f||^2 - ||g||^2$. We can recover $p(z, {\overline w})$ 
from $p(z, {\overline z})$  by polarization.

We can always bihomogenize $p$ by adding the variable $z_{n+1}$ and its conjugate. Thus we will usually assume that $p$
is bihomogeneous on ${\bf C}^{n+1} \times {\bf C}^{n+1}$. We say that $p$ is bihomogeneous of degree $(m,m)$.
It follows from (4) that $f$ and $g$ are homogeneous of degree $m$.
We will then regard the mapping $f \oplus g$ both as a polynomial map between complex Euclidean spaces
and as a rational holomorphic mapping between complex projective spaces:

$$ f \oplus g : {\bf P}^n \to {\bf P}^{A+B-1}. $$
Consider the hyperquadric $Q(A,B-1)$.
Suppose we are working in the open subset of ${\bf P}^{A+B-1}$ where $Z_l \ne 0$.
We divide by $|Z_l|^2$ in (3.2) for some $l$, and we obtain the usual defining equation of a hyperquadric in affine space.
The image of the subset of ${\bf C}^n$ defined by $p(z,{\overline z})=0$ under the map $f \oplus g$ lies in $Q(A,B-1)$. 

In order to better understand this mapping we naturally study the ideal $I(p)$ generated by $p$. 
The collection ${\mathcal H}(n)$ of Hermitian symmetric polynomials on ${\bf C}^n$ is 
isomorphic to the (real) polynomial ring in $2n$ real variables.
We  introduce
the subset ${\mathcal H}(n;A,B)$ of ${\mathcal H}(n)$ consisting of polynomials with signature pair $(A,B)$.  Given $p$,
one of our main concerns will be the set $I(p) \cap {\mathcal H}(n;A,B)$.
 For any ideal $J$ we may also consider the collection of pairs $(A,B)$ 
for which there is 
an element $q \in J$
for which ${\bf s}(q) = (A,B)$, or equivalently, if $J \cap 
{\mathcal H}(n;A,B)$ is non-empty. For each ideal,  $(0,0)$ is in this 
collection.

Given $p$ with ${\bf s}(p) = (A,B)$,
as above we obtain a holomorphic polynomial mapping to the hyperquadric $Q(A,B-1)$. Let us now make the connection to a basic question
in CR Geometry; what are the CR mappings from the unit sphere to a given hyperquadric?
Put $r(z,{\overline z}) = ||z||^2 -|z_{n+1}|^2$. If $I(r) \cap {\mathcal H}(n;A,B)$ is non-empty,
then we can find a multiple $q$ of $r$ with ${\bf s}(q) = (A,B)$. As discussed above, $q$ determines a rational mapping from 
the unit sphere to the hyperquadric $Q(A,B-1)$. Restrictions on possible target hyperquadrics for rational mappings from spheres
(satisfying additional properties such as degree bounds or group-invariance) lead to the fundamental issues in CR complexity theory.

To study CR complexity, we need to understand how the signature pair behaves under multiplication.
Consider the (ordinary) product of two Hermitian symmetric polynomials $p$ and $q$
with ${\bf s}(p) = (A,B)$ and ${\bf s}(q)= (C,D)$. Then ${\bf s}(pq) = (E,F)$, where $E \le AC+BD$ and $F \le AD+BC$. In
the generic case we have $E=AC+BD$ and $F=AD+BC$, but strict inequalities can arise. In fact,
both elements in the ordered pair can decrease and thus ${\bf r}(pq) < {\rm min}({\bf r}(p), {\bf r}(q))$ is possible.
We call this situation {\it collapsing} of the rank.

The next result shows that we cannot collapse the rank to $1$; in other words, if the rank of a product is $1$,
then each factor must have rank $1$. Later we will see that there exist Hermitian symmetric polynomials
of arbitrarily large rank whose product has rank $2$.
Also the conclusion of Proposition 2.2 fails for real-analytic Hermitian symmetric functions. Consider the identity $1 = e^{||z||^2} e^{-||z||^2}$.
If we expand the exponential as a series, then the signature pairs of the factors would be $(\infty,0)$ and $(\infty, \infty)$. Yet their product
has signature pair $(1,0)$. 

\begin{proposition} Let $p$ and $q$ be Hermitian symmetric polynomials with ${\bf r}(pq)=1$.
Then ${\bf r}(p) = {\bf r}(q) = 1$. \end{proposition}
\begin{proof} We may assume that $pq= |h|^2$ for a complex valued holomorphic polynomial $h$.
We need to show that each of $p$ and $q$ is itself (plus or minus) the squared absolute value of a holomorphic polynomial.
Polarize the relation $pq=|h|^2$ to obtain
$$ p(z, {\overline w}) q(z, {\overline w}) = h(z) {\overline {h(w)}}. \eqno (5) $$
Now factor $h$ into irreducibles. Each irreducible factor $\phi(z)$ of $h(z)$
must divide either $p(z,{\overline w})$ or $q(z,{\overline w})$. Write $h(z) = u(z)v(z)$ where $u(z)$ divides $p(z,{\overline w})$
and $v(z)$ divides $q(z,{\overline w})$. By symmetry the same is true
for the functions of $w$. We may therefore divide both sides of (5) by $h(z){\overline {h(w)}}$. Put back $w=z$. We obtain 
$$ 1= {p \over |u|^2} \ {q \over |v|^2}. \eqno (6) $$ 
Each factor in (6) is a polynomial; by (6) each must be constant. It follows that ${\bf r}(p) = {\bf r}(q) = 1$. \end{proof}

\section{Factorizations}

 Many of the phenomena which arise
are consequences of elementary factorizations of polynomials in one real variable. We gather these results
in this section.

\begin{lemma} In (7.3) set $a= \sqrt{ 4 \pm 2 \sqrt{2}}$. In (7.4) set $b = \sqrt{2}$. The following identities hold:
$$ 1+ t^4 = (t^2 + \sqrt{2} t + 1) (t^2 - \sqrt{2}t + 1), \eqno (7.1) $$

$$ 1- t^6 = (1+t-t^3-t^4)(1-t+t^2), \eqno (7.2) $$

$$ 1 + t^8 = (t^4 + a t^3 + {a^2 \over 2} t^2 +  a t + 1) (t^4 - a t^3 + {a^2 \over 2} t^2 - a t + 1), \eqno (7.3) $$

$$ 1 + t^{12} = (1 - b t + t^2) (1 + b t + t^2 - t^4 - b t^5 - t^6 + t^8 + b t^9 + t^{10}), \eqno (7.4) $$

$$ 1-t^2-2t^6+t^7 = (1-t+t^2)(1+t-t^2-2t^3-t^4+t^5). \eqno (7.5) $$
\end{lemma}

The importance of these factorizations stems from the number of terms. Let $p$ be a polynomial
in one or several real variables $x$. We write ${\bf S}(p)=(a,b)$ if
$p$ has $a$ positive terms and $b$ negative terms. In each formula from
Lemma 3.1, let $p$ denote the left-hand side and let $p_1$ and $p_2$ denote the factors.
Then for example (7.1) illustrates ${\bf S}(p)= (2,0)$ whereas ${\bf S}(p_1)=(3,0)$ and ${\bf S}(p_2)=(2,1)$.
Also (7.2) illustrates ${\bf S}(p)= (1,1)$ whereas ${\bf S}(p_1)=(2,2)$ and ${\bf S}(p_2)=(2,1)$.

\begin{lemma} Assume $m \ge 2$. For $t \in {\bf R}$ put $P(t) = t^{2^m} + 1$. Thus ${\bf S}(P)=(2,0)$.
Then there is a polynomial $Q(t)$ such that
\begin{itemize}
\item All $2^{m-1}+ 1$ coefficients of $Q$ are positive. Thus ${\bf S}(Q)=(2^{m-1}+1,0)$.

\item $P(t) = Q(t) Q(-t).$
\end{itemize} \end{lemma}
\begin{proof} Regard $t$ as a complex variable. Put $\omega = e^{ \pi i \over 2^m}$.
The roots of $P$ occur when $t$ is a $2^m$-th root of $-1$, and hence are odd powers of $\omega$. 
Factor $P$ into linear factors:
$$ P(t) = \prod_j (t-\omega^{2j+1}). $$
The roots are symmetrically located in the four quadrants. 
We define $Q$ by taking the product over the terms where ${\rm Re}(\omega^{2j+1}) < 0$.
Each such factor has a corresponding conjugate factor. Hence
$$ Q(t) = \prod (t^2 - 2 {\rm Re}(\omega^{2j+1})t + 1), $$
and all the coefficients of $Q$ are positive. The remaining terms in the factorization of $P$ define $Q(-t)$, 
and the result follows. \end{proof}

Note that (7.1) and (7.3) are special cases of Lemma 3.2. Our next 
factorization will not be explicit; these polynomials
arise in [D1] and [DLP].

\begin{lemma} Let $f_d(x,y)$ be the polynomial
$$ f_{d} (x,y) = \left( {x + \sqrt{x^2 + 4y} \over 2}\right)^d + \left( {x - \sqrt{x^2 + 4y} \over 2}\right)^d + (-1)^{d+1} y^d. \eqno (8) $$
For each positive integer $d$, $f_d-1$ is divisible by $x+y-1$.
\end{lemma}

\section{Signature pairs of products}

We first establish a link between real polynomials and Hermitian symmetric polynomials whose corresponding Hermitian form
is diagonal. The matrix of coefficients of a Hermitian symmetric polynomial $p$ 
is diagonal if and only if we can write $p= ||f||^2 - ||g||^2$ where the components of $f\oplus g$ are monomials.
Consider the {\it moment map}
$$ z \to x = (x_1,...,x_n)= (|z_1|^2,...,|z_n|^2) = {\bf m}(z). \eqno (9) 
$$
Given a Hermitian symmetric polynomial $p$ whose underlying matrix of coefficients is diagonal,
there is a (real) polynomial $P$ in $x$ such that
$$ p(z,{\overline z}) = P(x)= P({\bf m}(z)). $$
We write $p = P \circ {\bf m}$. It is
 evident that ${\bf S}(P) = {\bf s}(p)$, and we hence also  call
${\bf S}(P)$ the signature pair of $P$. 
For some problems in Hermitian linear algebra, the diagonal
case describes the general case; such a statement applies in Theorem 4.1.
See also [DKR] and [DLP] for many additional results in and uses of this special situation.

Given a signature pair $(A,B)$
we wish to find  $p$ with ${\bf s}(p) = (A,B)$ and satisfying some other 
properties. Often we may do so by  first finding a real polynomial
$P$ such that ${\bf S}(P)= (a,b)$.  Then we put $p = P \circ {\bf m}$ or 
$p = H(P \circ {\bf m})$, where
$H$ denotes bihomogenization. In either case ${\bf s}(p)=(a,b)$.

We give a simple example of this sort of reasoning. The signature pair of $r$ is $(N,0)$ if and only if $r$ is the squared norm $||f(z)||^2$ of a
holomorphic polynomial mapping $f$ with $N$ linearly independent components.
It is possible to square a Hermitian symmetric polynomial, which is not itself a squared norm,
and obtain a squared norm. We give an example from [D3] and [DV] where ${\bf s}(r^2) = (9,0)$ whereas ${\bf s}(r)= (4,1)$.

\begin{example} For $0 < \epsilon <1$ put $P_\epsilon(t) = (1+t)^4 - (6 + \epsilon) t^2$. Note that ${\bf S}(P_\epsilon)=(4,1)$
and that ${\bf S}(P_\epsilon^2)= (9,0)$. Put
 $r_\epsilon = H(P_\epsilon \circ {\bf m})$. Then $r_\epsilon$ is defined 
by 
$$ r_\epsilon(z,{\overline z}) = (|z_1|^2 + |z_2|^2)^4 - (6+\epsilon) |z_1 z_2|^4. $$
For $0 < \epsilon < 1$ we have ${\bf s}(r_\epsilon ^2) = (9,0)$ whereas ${\bf s}(r_\epsilon)= (4,1)$.
\end{example}

The authors do not know whether there is an $r$ with ${\bf s}(r) = (3,1)$
and ${\bf s}(r^2) = (N,0)$ for some $N$. There is no $r$ with ${\bf s}(r) = (2,1)$ and
${\bf s}(r^2) = (N,0)$ for some $N$.

We have seen how to collapse terms via multiplication of polynomials in one variable.
This process of composition with the moment map leads to similar phenomena in the Hermitian symmetric case.

\begin{proposition} There are Hermitian symmetric polynomials
$q$ and $r$ such that the following hold:
\begin{itemize} 
\item $q$ and $r$ are each bihomogeneous of degree $(2^{m-1}, 2^{m-1})$. 

\item ${\bf s}(q) = (2^{m-1} + 1, 0)$.

\item ${\bf s}(r) = (2^{m-2} +1, 2^{m-2})$.

\item ${\bf s}(qr) = (2, 0)$.
\end{itemize}
\end{proposition}

\begin{proof} For each integer $m$ with $m\ge 2$, consider the polynomial

$$ p(z,{\overline z}) = |z_1|^{2^m} + |z_2|^{2^m}. $$
Then ${\bf s}(p) = (2,0)$. 
Set $x=(x_1,x_2) = (|z_1|^2, |z_2|^2)$. Then $p(z) = x_1^{2^m} + x_2^{2^m}$.
We factor $p$ by first setting $x_2=1$ and then using
Lemma 3.2. The polynomial $Q$ in that lemma has all positive coefficients
and thus ${\bf S}(Q) = (2^{m-1}+1,0)$. Set $q = Q \circ {\bf m}$.
After homogenizing we obtain the desired conclusion.
\end{proof}
It follows that the rank ${\bf r}(qr)$ of a product can be $2$ even when
the ranks of the factors are  arbitrarily large. It is also possible to 
create examples 
where the rank of the product is $2$ and neither factor has signature pair $(k,0)$.
The examples in the next proposition will be used to complete the story in Theorem 4.1.

\begin{proposition} There are Hermitian symmetric polynomials
$q$ and $r$ such that the following hold:
\begin{itemize} 

\item ${\bf s}(q)= (2,2)$, ${\bf s}(r) = (2,1)$, and ${\bf s}(qr) = (1,1)$.
\item ${\bf s}(q)= (5,0)$, ${\bf s}(r) = (3,2)$, and ${\bf s}(qr) = 
(2,0)$.
\item ${\bf s}(q)= (6,3)$, ${\bf s}(r) = (2,1)$, and ${\bf s}(qr) = 
(2,0)$.
\item ${\bf s}(q)= (3,3)$, ${\bf s}(r) = (2,1)$, and ${\bf s}(qr) = (2,2)$.
\end{itemize}
\end{proposition}
\begin{proof} Let $Q$ and $R$ denote the factors of any of the real 
polynomials appearing in
(7.2), (7.3), (7.4), and (7.5). In each case let $q$ and $r$
be the 
bihomogenizations of $Q \circ {\bf m}$ and $R \circ {\bf m}$.
We obtain the desired examples. \end{proof}

Consider the formula for $f_d$ from Lemma 3.3. In it we replace $x$ by $|z_1|^2$, replace $y$ by $|z_2|^2$, subtract $1$, 
and homogenize using the $z_3$ variable. 
We obtain a bihomogeneous polynomial
$p_d$ of degree $(d,d)$ which vanishes on the set $r=0$, where
$r=|z_1|^2+|z_2|^2 - |z_3|^2$. Then $p_d=q_dr$. For $d=2m+1$, we have
${\bf s}(p_d) = (m+2,1)$. When $d=2m$ we have ${\bf s}(p_d) = (m+1, 2)$. Thus ${\bf r}(p_d) = m+3$. One can show
that ${\bf r}(q_d) = {d(d+1) \over 2}$, that is, $q_d$ has {\it maximum} rank for its degree. Yet (by [DKR]),
$p_d=q_d r$ has {\it minimum} rank for its degree, given that $p_d \in {\mathcal H}(2;N,1) \cap I(r)$
and that it depends only on the squared moduli of the variables.

These polynomials have other important properties. For example, they 
are invariant under an interesting representation of a finite cyclic subgroup of the unitary group $U(2)$. See [D1] and [D4].

We have seen that it is subtle to decide what happens to the ranks of Hermitian symmetric 
polynomials under multiplication. The next Theorem completely answers what are the possible signature
pairs of a non-trivial product.

The crucial information in the statement of Theorem 4.1 below
is that $r_1$ and $r_2$ are indefinite.  By Proposition 4.1  we can obtain 
$(2,0)$ for the signature pair
of a product when one of the factors has signature pair $(A,0)$.
What happens if we insist that neither factor is positive semi-definite,
in other words, that neither factor is a squared norm? Remarkably, we can still get $(2,0)$.
In fact we can get any pair except $(0,0)$ (obviously), $(1,0)$, or $(0,1)$ (by Proposition 2.2).

\begin{theorem} There exist indefinite polynomials $r_1$ and $r_2$ such that ${\bf s}(r_1 r_2) = (A,B)$
if and only if $(A,B)$ is not one of $(0,0)$, $(1,0)$, $(0,1)$. \end{theorem}

\begin{proof} First we consider the three exceptional pairs. The signature pair is $(0,0)$ only if $r=0$.
An indefinite polynomial is not the zero polynomial. Since the product of nonzero polynomials is nonzero, this case is trivial. 
Proposition 2.2 shows that a product can have rank $1$ only if each factor does. Since an indefinite polynomial must have rank at least $2$,
the cases $(1,0)$ and $(0,1)$ are also ruled out. 

Conversely suppose that $(A,B)$ is not one of these three pairs.
We will provide explicit formulas, writing $t= x_1 = |z_1|^2$ and putting $|z_2|^2=1$.
Our examples will be expressed as polynomials in one real variable; in each case we compose with the moment map $m$
as above and set $p = P \circ {\bf m}$. Then ${\bf S}(P) = {\bf s}(p)$. 

First we consider pairs of the form $(N,0)$. Put $a = \sqrt{2}$. By (7.4) we have

$$ 1+ t^{12} = (1 - a t + t^2) (1 + a t + t^2 - t^4 - a t^5 - t^6 + t^8 + a t^9 + t^{10}) = r_1(t) r_2(t). \eqno (10) $$
In (10) the factors of $1+t^{12}$ have
signature pairs $(2,1)$ and $(6,3)$, and hence each defines an indefinite form.
Their product determines a Hermitian form with signature pair $(2,0)$.

To get an example where the product has signature pair $(n+2,0)$, we first define
$r_1$ and $r_2$ as in (10) and then write

$$ (1+t^{12})^{n+1} = (1+t^{12})^n r_1(t) r_2(t) = q_1(t) r_2(t). \eqno (11) $$
The left-hand side of (11) has signature pair $(n+2,0)$. On the other hand the first
two terms in the expansion of $q_1(t) = (1+t^{12})^n r_1(t)$ are $1 - a t$, and hence $q_1$ has terms of both signs.
Thus both $q_1$ and $r_2$ determine indefinite forms, and yet ${\bf S}(q_1 r_2) = (n+2,0)$.

Next, for $\epsilon >0$ to be chosen, consider the polynomial
$$ r(t) = (1-t)^2 (1+ \epsilon t)^n = (1-t) r_2(t). \eqno (12) $$
Note that $r$ is of degree $n+2$. If $\epsilon$ is
 sufficiently small, then all of the coefficients of $r$
are non-zero, and only the coefficient of $t$ is negative.
 Hence ${\bf S}(r) = (n+1,1)$. The factor $1-t$ 
corresponds to an indefinite form;
for $n\ge 1$ the factor
 $r_2$ given by $r_2(t) = (1-t)(1+\epsilon t)^n$ also corresponds to an
 indefinite form when $n \epsilon < 1$. Thus we have a product
of indefinite forms whose signature pair is $(k,1)$ for $k \ge 2$.

We obtained the pair $(1,1)$ in Proposition 4.2 via formula 
(7.2).

$$ 1- t^6 = (1+t-t^3-t^4)(1-t+t^2). \eqno (13) $$
The signature pair of the product
is $(1,1)$ and each factor is indefinite. 

We next obtain products of indefinite factors
 such that the signature pair of the product is $(A,B)$ where $A\ge B \ge 2$.
Because we can multiply by $-1$, we also obtain such products with signature pairs $(A,B)$ with $B \ge A \ge 2$.
Define $r_1(t) = 1 \pm t + \pm t^2 + ...+ t^d$, where at least one sign is negative. If there are $a$ positive terms,
we have ${\bf S}(r_1) = (a,d+1-a)$ where $2 \le a \le d$.
Put $r_2(t) = 1 - \epsilon t$ for $\epsilon > 0$.  Then both $r_1$ and $r_2$ are indefinite.
If we choose $\epsilon$ sufficiently small, then we easily see that
${\bf S}(r_1 r_2) = (a,d+2-a)$. Hence we can achieve all signature pairs $(A,B)$ with $A\ge B \ge 2$ in this fashion.

We have now obtained all pairs except for the three exceptional cases considered in the first paragraph.
\end{proof}

One can obtain a deeper understanding of this result by introducing the inertia triple. To obtain $(2,0)$ for the signature
pair of a product of indefinite polynomials requires allowing enough vanishing eigenvalues.
Let us include the number of zero eigenvalues
in the notation. Then for example the inertia triple of the polynomial $|z_1|^{24} + |z_2|^{24}$
determined by $1+t^{12}$ is $(2,0,11)$; there are $11$ vanishing eigenvalues. We cannot find indefinite factors of polynomials with inertia
triple $(2,0,k)$ unless $k$ is large enough. By (10)  
it is possible when $k=11$. It is obviously
impossible when $k$ is small. We do not investigate this matter in this paper.

We do wish to provide an example where
the general case differs from the diagonal case regarding inertia triples.
We  give indefinite $p$ and $q$, with ${\bf in}(p)=(2,1,0)$, ${\bf in}(q)=(3,3,0)$, and ${\bf in}(pq) = 
(2,2,6)$.  

$$ p = z_1\bar{z}_3+z_2\bar{z}_2 + z_3\bar{z}_1 = |z_2|^2 + |{z_1 + z_3 \over \sqrt{2}}|^2 -  |{z_1 - z_3 \over \sqrt{2}}|^2. $$

$$ q = z_1^2\bar{z}_3^2 + z_2^2\bar{z}_2^2 + z_3^2\bar{z}_1^2 - z_1z_3\bar{z}_1\bar{z}_3 - z_1z_2\bar{z}_2\bar{z}_3 - z_2z_3\bar{z}_1\bar{z}_2. $$

$$ pq = z_1^3 \bar{z}_3^3 + z_2^3\bar{z}_2^3 + z_3^3\bar{z}_1^3 - 3z_1 z_2  z_3 \bar{z}_1 \bar{z}_2 \bar{z}_3 $$
$$ = |z_2^3|^2 + |{z_1^3 + z_3^3 \over \sqrt{2}}|^2 -|{z_1^3 - z_3^3 \over \sqrt{2}}|^2  - |\sqrt{3} z_1z_2z_3|^2. \eqno (14) 
$$

This  example is (up to a linear change of variables) the 
only one of its type possible. Notice, given the dimension, that the inertia triple specifies the  
degree.  Therefore, assuming that ${\bf in}(p) = (2,1,0)$, $p$ has to be
the defining equation for a sphere and $pq$ has the signature for a defining equation of
the hyperquadric $Q(2,1)$.  The second author verified in [L] that there is 
(up to a linear change of coordinates)  only one such map, and therefore it 
must be the one above.  Furthermore  $pq$ 
cannot be made  diagonal after any  linear change of 
coordinates. Therefore, if we restrict to diagonal forms, no such $p$ and $q$ exist.

On the other hand, if we  give no  restriction on the number of zero
eigenvalues,  many different examples (including diagonal ones) are 
possible.  As in Theorem 4.1  we can use  polynomials in one variable.
If we want to work in the
same dimension, in ${\mathbb P}^2$, we simply  assume that  the
example does not depend on the other variable.  In the  following
 example,  ${\bf S}(P) = (2,1)$, ${\bf S}(Q)=(3,3)$, and ${\bf 
S} (PQ)=(2,2)$.
$$ 
1-t^2-2t^6+t^7 = (1-t+t^2)(1+t-t^2-2t^3-t^4+t^5) $$
For the corresponding $p$ and $q$ in ${\mathbb P}^1$ we have
${\bf in}(p)=(2,1,0)$, ${\bf in}(q)=(3,3,0)$, and ${\bf in}(pq)=(2,2,4)$.
In ${\mathbb P}^2$ we have ${\bf in}(p)=(2,1,3)$, ${\bf in}(q)=(3,3,15)$, 
and ${\bf in}(pq)=(2,2,32)$.

We close this section by interpreting a result from [CD] in the language of inertia triples.
Let  $p$ be a bihomogeneous polynomial and assume that $p(z,{\overline z}) > 0$ on the unit sphere.
Then there is an integer $d$ such that $||z||^{2d} p(z, {\overline z})$ has inertia triple equal to
$(N,0,0)$. In other words, the product is a squared norm and its underlying Hermitian form is strictly
positive definite; there are no vanishing eigenvalues.

\section{CR Mappings}

We gather some remarks connecting our work with the study of CR mappings.
Let $f = {g \over h}$ be a rational mapping reduced to lowest terms.  Assume $f : S^{2n-1} \to 
S^{2N-1}$. 
We form the Hermitian symmetric polynomial $p = ||g||^2 - |h|^2$. Note that $f$ is constant
if and only if $p=0$. Note that $p$ vanishes on the sphere.  We homogenize 
using the $z_{n+1}$ variable
to obtain homogeneous holomorphic polynomials $g^*$ and $h^*$
 and a bihomogeneous $p^*$ such that

$$ p^*= ||g^*||^2 - |h^*|^2. \eqno (15) $$
Also, $p^*$ is divisible by $r$, where $r= ||z||^2 - |z_{n+1}|^2$.
If we assume that $h$ and the components of $g$ are linearly independent, then ${\bf s}(p^*) = (N,1)$ and this 
rational map $f : S^{2n-1} \to S^{2N-1}$ shows that $I(r) \cap {\mathcal H}(n;N,1)$ is non-empty.

Conversely, for given holomorphic homogeneous polynomials, suppose (15) holds and $p^*$ is divisible by $r$.
Set $z_{n+1}$ equal to $1$ to dehomogenize. Then ${g \over h}$ maps the unit sphere in its domain
to the unit sphere in its target. The trivial case when $f$ is a constant is ruled out by the assumption of linear independence.
For $n\ge 2$, the maximum principle guarantees that $h$ does not vanish on the closed ball,
and $f$ defines a rational proper mapping between balls. See Proposition 
5.1. We also mention, for $n\ge 2$,
that a sufficiently smooth CR mapping between spheres $S^{2n-1}$ and $S^{2N-1}$ must be the restriction
of a rational mapping. [F].

More generally, we can allow the target to be a hyperquadric. In this case there exist smooth CR mappings that are not rational.
An additional new phenomenon must be noted.
Consider a holomorphic mapping 
$(f,g)$ which maps the sphere to a hyperquadric. The analogue of (15) becomes
$$ ||f(z)||^2 - ||g(z)||^2 - 1. \eqno (16) $$
Without any assumption on linear independence, cancellation can occur in the squared norms in (16). We give an example when $n=2$
that illustrates the point. Consider the map $z \to \zeta(z) = (z_1,z_2, w(z), w(z))$. Then 

$$ |\zeta_1|^2 + |\zeta_2|^2 + |\zeta_3|^2 - |\zeta_4|^2 - 1 $$
vanishes identically on the sphere. Hence we obtain a holomorphic mapping from the sphere
to a hyperquadric. Since the holomorphic function $w$ is otherwise 
arbitrary, we have no control on the map $\zeta$.
By assuming that the components are linearly independent we eliminate this difficulty.

Without loss of generality we may assume that the components of a rational map are linearly independent.
The problem of describing rational holomorphic mappings from the sphere $S^{2n-1}$ 
to a hyperquadric $Q(A,B-1)$ is equivalent to analyzing the sets $I(r) \cap {\mathcal H}(n;A,B-1)$. 
The following well-known result applies when the target is a sphere.

\begin{proposition} Assume that there exists a non-constant Hermitian symmetric homogeneous polynomial $p$ on ${\bf C}^{n+1}$
such that ${\bf s}(p)=(k,1)$ where $k \ge 1$, and such that $p$ vanishes on the set $||z||^2 = |z_{n+1}|^2$.
Equivalently, suppose that $h=(h_1,...,h_k)$ is
 a non-constant rational mapping on ${\bf C}^n$,
 with linearly independent 
components and with  $||h(z)||^2 =1$ on the unit sphere. Then $k\ge n$. 
\end{proposition}
\begin{proof}  First we observe that the two statements are equivalent. 
The first statement guarantees that there is a polynomial map $f:{\bf C}^{n+1} \to {\bf C}^k$ and a polynomial $g$ on ${\bf C}^{n+1}$
such that $p= ||f||^2 - |g|^2$. Divide by $g$ and set $z_{n+1} =1$ to dehomogenize. We obtain a rational mapping $h$ 
sending the unit sphere $S^{2n-1}$ to $S^{2k-1}$. Conversely, given $h$ we let $G$ be the least common denominator
of its components and we write $h = {F \over G}$. After homogenization of $F$ and $G$ to $f$ and $g$, we 
put $p=||f||^2 - |g|^2$. Thus the two statements are equivalent.

To show that $k \ge n$ we prove an apparently stronger assertion. Suppose $h=(h_1,...,h_k)$ is a holomorphic mapping defined
near the unit sphere, and $||h(z)||^2 =1$ on the unit sphere $S^{2n-1}$. If $n=1$, there is nothing to prove.
If $n \ge 2$, then either $h$ is a constant or
$h$ extends to a proper holomorphic mapping from the unit ball to the unit ball. We are assuming that $h$ is not constant.
If $h$ is proper holomorphic, then the inverse image of a point is both
a compact set and a complex variety. Hence the target dimension $k$ cannot be smaller than the domain dimension $n$.
\end{proof}

We recall the notion of {\rm degree estimate} for rational mappings between spheres in preparation
for our applications to CR geometry. See Corollary 8.2 for the situation 
when the target is a hyperquadric.

\begin{definition} A {\rm degree estimate} for rational mappings between spheres
is a number $c(n,N)$ with the following property. Whenever $f:S^{2n-1} \to S^{2N-1}$ is rational
and of degree $m$, then $m \le c(n,N)$. \end{definition}

No such estimates hold when $n=1$. Otherwise such estimates do hold, 
but {\it sharp} degree estimates are not known in general. See [DKR], [DLP], [DL], [LP], [M] for various results in this direction.
We mention here only the following degree estimate from [DL]; for $n \ge 2$, the degree of a rational mapping from $S^{2n-1}$ to
$S^{2N-1}$ is bounded above by ${N(N-1) \over 2(2n-3)}$.

\section{projective degree}

Suppose that $p$ is bihomogeneous of degree $(m,m)$.
Let $J={\mathcal H}(n;A,B) \cap I(p)$ and assume that $J$ is non-empty. To study the complexity of $J$ we are led to the notion
of projective degree. Given two Hermitian symmetric polynomials $p$ and $q$, we say that they are equivalent, writing $p \sim q$,
if their ratio equals $|h(z)|^2$ for some rational $h$. Note that $h$ is independent of ${\overline z}$. Thus

$$ p \sim q \iff |u(z)|^2 p(z, {\overline z}) = |v(z)|^2 q(z, {\overline z}), $$
for holomorphic polynomials $u$ and $v$.

\begin{definition} The {\it projective degree} of $p$ is the smallest 
$m$ such that there is a bihomogeneous polynomial of degree $(m,m)$ equivalent to $p$.
Equivalently, it is the degree of the rational map $f\oplus g :{\bf P}^n \to {\bf P}^{A+B-1}$ induced by $p$. 
We write ${\mathcal D}(p)$ for the projective degree. \end{definition}

\begin{definition} Let $S$ be a non-empty set of Hermitian symmetric polynomials. 
We write ${\mathcal D}^*(S)$   for ${\rm sup}_{p \in S} {\mathcal D} (p)$.
 \end{definition}

To find the projective degree of the bihomogeneous polynomial $p = ||f||^2 - ||g||^2$,
we divide out all common factors of the form $|h|^2$, for $h$ a holomorphic homogeneous polynomial.
The result is bihomogeneous of degree $(m,m)$ where $m$ is the projective degree of $p$.

Consider Hermitian symmetric polynomials taking the value $1$ on the unit sphere. 
Certain signature pairs allow for examples of arbitrary degree, other signature pairs allow only examples whose degrees
form a bounded set,  and some
signature pairs rule out all examples. See Theorem 8.2 for a general 
result along these lines.

\begin{example} Consider Hermitian symmetric polynomials $p$,$q$, and $r$  given by
$$ p(z, {\overline z}) = |z_1|^{2m+2} + |z_1|^{2m} |z_2|^2 -|z_1|^{2m}|z_3|^2  = |z_1|^{2m} (|z_1|^2 + |z_2|^2 -|z_3|^2), 
\eqno (17.1) $$
$$ q(z, {\overline z}) = |z_1|^{2m+2} + |z_1|^{2m} |z_2|^2 +|z_3|^{2m+2} - |z_1|^{2m} |z_3|^2, \eqno (17.2) $$
$$ r(z, {\overline z}) = {-|z_3|^2 \over 2} + {3(|z_1|^{2} + |z_2|^2) \over 2}. \eqno (17.3) $$
Note that $p=0$ on the unit sphere and $q=r=|z_3|^2$ on the unit sphere.
The projective degrees of $p$ and $r$ equal $2$, and the projective degree of $q$ is $2m+2$. 
We have ${\bf s}(p)= {\bf s}(r) = (2,1)$ and ${\bf s}(q)=(3,1)$.
Thus there is a Hermitian polynomial of arbitrarily large
projective degree that is identically one on the sphere and has signature pair $(3,1)$. If a function
is identically 1 on the sphere and has signature pair $(2,1)$, however, then its projective degree is at most $2$.

\end{example}

\section{Domain dimensions one and two}

We consider rational maps from  $S^{2n-1}$ to a hyperquadric.
Given the signature pair for the target 
hyperquadric we want to know whether non-trivial maps exist; 
if so, we ask how complicated they can be.  
Recall that $$ r(z,{\overline z}) = ||z||^2 - |z_{n+1}|^2 = \sum_{j=1}^n |z_j|^2 - |z_{n+1}|^2. $$
In order to prepare for the general stability result,
 we will prove a stronger statement, Theorem 7.1, in two dimensions.
For completeness we also consider the one-dimensional case.

Given $n,A,B$ we want to understand the sets $J={\mathcal H}(n;A,B) \cap I(r)$.  
In all dimensions ${\mathcal H}(n;0,0) \cap I(r) = \{0\}$. 
The only significant difference when $n=1$ is that this set is the {\it only} non-empty such $J$ for which 
${\mathcal D}^*(J)$ is finite.

\begin{proposition} Put $n=1$. The following assertions hold:
\begin{itemize} 
\item ${\mathcal H}(1;0,0) \cap I(r) = \{0\}$.

\item ${\mathcal H}(1;k,0) \cap I(r)$ and ${\mathcal H}(1;0,k) \cap I(r)$ are empty for all $k$.

\item In all other cases ${\mathcal D}^*\left( {\mathcal H}(1;a,b) \cap I(r) \right)$ is infinite.
\end{itemize}
\end{proposition}
\begin{proof} The first two statements are trivial. The third statement is easy.
First note that $|z_1|^{2m}- |z_2|^{2m} \in {\mathcal H}(1;1,1) \cap I(r)$ and its projective degree is $m$.
We then easily write down bihomogeneous polynomials $p$ of arbitrary degree in $I(r)$ with ${\bf s}(p) = (a,b)$
for any pair $(a,b)$ of positive integers.
\end{proof}

In this section we henceforth assume $n=2$ and thus $r = |z_1|^2 +|z_2|^2 - |z_3|^2$.

\begin{lemma} For each integer $k$ with $k \ge 2$, the set ${\mathcal H}(2;k,1) \cap I(r)$ is nonempty. \end{lemma}
\begin{proof} Assume $d$ is a positive integer.
Consider the real {\it Whitney} polynomials used in [D1] and [DLP] defined by
$$ W_d (x,y) = x^d + x^{d-1} y + x^{d-2} y + ... + y. $$
We note that $W_d(x,1-x)=1 $ for all $x$, because the finite geometric series gives:
$$ W_d(x,y) = x^d + y { 1-x^d \over 1-x}. $$
Let $w_d(x,y,\zeta)$ be the homogenized version of $W_d -1$. Then $w_d$ vanishes on the set $x+y - \zeta = 0$,
and hence there is a polynomial $g(x,y,\zeta)$ such that
$$ w_d(x,y,\zeta) = g(x,y,\zeta) (x+y- \zeta). \eqno (18) $$
Set $x=|z_1|^2$, $y=|z_2|^2$, and $\zeta = |z_3|^2$ in (18).
The left-hand side of (18) becomes a Hermitian symmetric polynomial $p$ lying in $I(r)$. 
Furthermore its underlying Hermitian matrix is diagonal
and has precisely $d+1$ positive eigenvalues and one negative eigenvalue. Thus ${\bf s}(p) = (d+1,1)$.
Put $k=d+1$ to complete the proof. \end{proof}

\begin{corollary} Let $k\ge 2$. For each pair $(a,b)$ of nonnegative integers
consider the pairs 
$$ (A,B) = (k,1) + a(2,1) + b(1,2). \eqno (19.1) $$
$$ (A,B) = (1,k) + a(2,1) + b(1,2). \eqno (19.2) $$
For all such $(A,B)$ the set $J={\mathcal H}(2;A,B) \cap I(r)$ is non-empty.
Furthermore, if either $a$ or $b$ is positive, then ${\mathcal D}^*(J) = \infty$.
(In other words, there are elements in $J$ of arbitrarily large projective degree.)
\end{corollary}

\begin{proof} Let $p$ be a Hermitian symmetric polynomial.
Assume that $p= rv$ and ${\bf s}(p) = (K,L)$. Choose an integer $m$ much larger than the degree of $v$.
Define $q$ by 
$q = r (Hv \pm |z_1|^{2m})$,
where $Hv$ denotes the homogenization of $v$. Because of the restriction on degree, there is no cancellation or collapse of rank.
We see that ${\bf s}(q) = (K+2,L+1)$ if we choose the plus sign and ${\bf s}(q)= (K+1,L+2)$
if we choose the minus sign. Thus, given any element in ${\mathcal H}(2;K,L) \cap I(r)$,
we can construct elements in both ${\mathcal H}(2;K+2,L+1) \cap I(r)$ and
${\mathcal H}(2;K+1,L+2) \cap I(r)$. Applying this procedure (with large enough degrees at each stage)
$a$ times using the plus sign and $b$ times using the minus
sign proves the corollary. \end{proof}

The only pairs $(A,B)$ with $A,B \ge 1$ not included in Lemma 7.1 and 
Corollary 7.1 are
$(1,1)$, $(2,2)$, $(3,2)$, $(2,3)$, $(4,4)$. It is evident that
${\mathcal H}(2;1,1)$ is empty. We wish to analyze the other cases. First we give
an instructive example.

\begin{example} Let $\lambda = (A,B,C)$ be a point in ${\bf R}^3$. Define a family of polynomials $h_\lambda(x,y,\zeta)$ given by
$$ h_\lambda (x,y,\zeta) =(Ax + By + C \zeta)(x+y-\zeta). \eqno (20) $$
As usual we set $x=|z_1|^2$, $y=|z_2|^2$, and $\zeta= |z_3|^2$ in (20)
to obtain homogeneous Hermitian symmetric polynomials $p_\lambda$. It follows 
from (20) that $p_\lambda \in I(r)$. 

It is easy to compute the signature ${\bf s}(p_\lambda)$
for different values of $\lambda$. For example, if $A=0$ and $0 < C < B$, we obtain
${\bf s}(p_\lambda) = (3,2)$. If $A=0$ and $0 < C = B$, then ${\bf s}(p_\lambda) = (3,1)$.
If $C > A \ge B > 0$, then ${\bf s}(p_\lambda) = (5,1)$.
If $A=C > B > 0$, then ${\bf s}(p_\lambda) = (4,1)$.
If $B > C > A > 0$, then ${\bf s}(p_\lambda) = (4,2)$. 
Replacing $\lambda$ by $-\lambda$ replaces $p_\lambda$ by $-p_\lambda$, and
replaces the signature pair $(a,b)$ with $(b,a)$.  
Not all pairs are possible; for example, one cannot get $(2,1)$ for any values of $\lambda$.

In particular we note that the sets ${\mathcal H}(2;A,B) \cap I(r)$ are nonempty when $(A,B)=(3,2)$ or when $(A,B)=(2,3)$. \end{example}

\begin{corollary} ${\mathcal H}(2;4,4) \cap I(r)$ is nonempty and ${\mathcal D}^*\left( {\mathcal H}(2;4,4) \cap I(r) \right)= \infty$. \end{corollary}
\begin{proof} Example 7.1 shows that ${\mathcal H}(2;3,2) \cap I(r)$ is 
nonempty. Corollary 7.1 then shows that ${\mathcal 
H}(2;4,4) \cap I(r)$ is nonempty and ${\mathcal D}^*\left( {\mathcal H}(2;4,4) \cap I(r)\right)= \infty$.
\end{proof}

\begin{example} ${\mathcal H}(2;2,2) \cap I(r)$ is nonempty.
Define a real polynomial $h$ by
$$ h(x,y) = x^2-y^2 - x + y = (x-y)(x+y-1).  \eqno (21) $$
As usual we bihomogenize $h \circ {\bf m}$ to  obtain  a Hermitian 
symmetric polynomial $p$ with ${\bf s}(p)= (2,2)$.
By (21),  $p$ lies in $I(r)$. For clarity we write the formula for $p$:

$$ p = (|z_1|^2 - |z_2|^2)( |z_1|^2 + |z_2|^2 - |z_3|^2). \eqno (22) $$ 
The signature pairs of the factors are $(1,1)$ and $(2,1)$ and yet ${\bf s}(p) =(2,2)$. 
\end{example}

The classification in [L], which relies on Faran's work [Fa1],
includes finding all elements of ${\mathcal H}(2;2,2) \cap I(r)$. 
In the notation of this paper, the main result in [F] implies that ${\mathcal D}^* \left( {\mathcal H}(2;3,1) \cap I(r) \right)=3$
and the main result in [L] states that ${\mathcal D}^* \left( {\mathcal H}(2;2,2) \cap I(r) \right)=3$.
Combining the results from [Fa1], [L] and this section
yields an almost complete answer in domain dimension two.

\begin{theorem} Let $r=|z_1|^2 + |z|^2 - |z_3|^2$ and let
 $J= {\mathcal H}(2;A,B) \cap I(r)$. Then 
\begin{itemize}
\item $J=\{0\}$ if $(A,B)=(0,0)$.

\item $J$ is empty 
\begin{itemize} 
\item if $(A,B) = (1,1)$, or
\item if $A=0$ or $B=0$ but $A+B \ne 0$. 
\end{itemize}
\item Let $(A,B)=(2,1)$ or $(1,2)$. Then ${\mathcal D}^*(J)=1$ and hence is finite. 
\item Let $(A,B)=(2,2)$, or $(3,1)$ or $(1,3)$. Then ${\mathcal D}^*(J)=3$ and hence is finite. 
\item Suppose $A, B \ge 2$ but $A+B \ge 6$. Then ${\mathcal D}^* (J)=\infty$.
\item Suppose $A=1$ or $B=1$ but $AB > 1$. Then ${\mathcal D}^*(J)$ is finite. 
\end{itemize}\end{theorem}

We do not know whether ${\mathcal D}^*(J)$ is finite in case $(A,B)$ is $(3,2)$ or $(2,3)$.
The following table summarizes the conclusions of the theorem. 
In this table, the numbers $0$, $1$ and $3$
denote known values of ${\mathcal D}^*(J)$. The symbol $-$ means
that ${\mathcal H}(2;A,B) \cap I(r)$ is empty. The letter $e$ indicates that the set is non-empty,
but we do not know whether ${\mathcal D}^*(J)$ is finite. The letter $f$ indicates that ${\mathcal D}^*(J)$ is finite.

\begin{center}
\begin{tabular}{c|c c c c c c c}
$\vdots$ & $\vdots$ & $\vdots$ & $\vdots$ & $\vdots$ & $\vdots$ & $\vdots$ & \\
5 & - & f & $\infty$ & $\infty$ & $\infty$ & $\infty$ & $\cdots$  \\
4 & - & f & $\infty$ & $\infty$ & $\infty$ & $\infty$ & $\cdots$  \\
3 & - & 3 & e & $\infty$ & $\infty$ & $\infty$ &  $\cdots$  \\
2 & - & 1 & 3 & e & $\infty$ & $\infty$ & $\cdots$  \\
1 & - & - & 1 & 3 & f & f & $\cdots$  \\
0 & 0 & - & - & - & - & - & $\cdots$ \\ \hline
\noalign{\smallskip}
${}^B/{}_A$ & 0 & 1 & 2 & 3 & 4 & 5 & $\cdots$
\end{tabular}
\end{center}

\section{Stability in higher dimensions}

We begin by stating a result from [DL]
 that provides the higher dimensional analogue of Lemma 7.1.

\begin{theorem} Put $T(n) = n^2 - 2n+2$. Assume $N \ge T(n)$.
Then there is a holomorphic polynomial mapping $f:S^{2n-1} \to S^{2N-1}$
for which $N$ is the minimum embedding dimension. \end{theorem}

\begin{corollary} For $N \ge T(n)$, the set
${\mathcal H}(n;N,1) \cap I(r)$ is nonempty. \end{corollary}

\begin{proof} Let $p$ be the bihomogenization of $||f||^2 - 1$. Since $f$ maps the sphere
to the sphere, $p \in I(r)$. Since $N$ is the minimum target dimension, ${\bf s}(p) = (N,1)$. \end{proof}

Let $p \in I(r)$ and suppose that $p$ is of degree $(d,d)$.  
What are the possible values of ${\bf s}(p)$?  Various results 
(for example, [BH], [Fa1], [Fa2], [HJ], [HJX]) provide
 information about ${\bf s}(p)$ in specific situations.
In Theorem 8.2 we prove a general {\it stability} result. 
We show that there is an integer $M$ with the following property.
Given a pair $(A,B)$ of non-negative integers,
if $A+B \ge M$, then we can decide
 whether $J = {\mathcal H}(n;A,B) \cap I(r)$ is non-empty, and if so, 
whether it has bounded projective degree.
For each lattice point $(A,B)$ with
 $A+B \ge N$ there are three possibilities:
no maps to a hyperquadric with that
signature exist, maps exist but there is a bound on their degrees, or maps of arbitrarily large degree exist.
When the degree is bounded we seek the largest degree ${\mathcal D}^*(J)$.

First we show how to construct new examples from given examples. The procedure here was used in the previous section in (19)
with the {\it change vectors} $(2,1)$ and $(1,2)$. Here we must use $(n,1)$ and $(1,n)$ as change vectors.

\begin{proposition} Suppose that $p= ||f||^2 - ||g||^2$ lies in ${\mathcal H}(n;A,B) \cap I(r)$.
Then the sets ${\mathcal H}(n;A+n,B+1) \cap I(r)$ and ${\mathcal H}(n;A+1,B+n) \cap I(r)$ are nonempty. 
Furthermore, in both cases, ${\mathcal D}^*(J) = \infty$. \end{proposition}

\begin{proof} Assume that $p$ is bihomogeneous of degree $(m,m)$. We write $p=rv$. Choose a monomial $w = w(z)$ 
of degree $k$, where $k > m+1$. Then bihomogenize $v \pm |w|^2$ to get $V$. Put $P=Vr$. Then $P$ is bihomogeneous of degree $(k,k)$ and it is
divisible by $r$. Because the new terms
do not interfere with the old terms, we have ${\bf s}(P) = (A+n,B+1)$ if we use $|w|^2$ and ${\bf s}(P) = (A+1, B+n)$
if we use $-|w|^2 $. \end{proof}

We now get to the main point. Given $n,A,B$,  is the set ${\mathcal H}(n;A,B) \cap I(r)$ non-empty,
and if so, is ${\mathcal D}^*\left( {\mathcal H}(n;A,B) \cap I(r) \right)$ finite?
For each $n\ge 2$, the result answers these questions for all but a finite number of pairs $(A,B)$.

\begin{theorem} Let $r(z,{\overline z}) = \sum_{j=1}^n |z_j|^2 - |z_{n+1}|^2$, and assume $n \ge 2$. Let $J= {\mathcal 
H}(n;A,B) \cap I(r)$
denote the collection of Hermitian symmetric bihomogeneous polynomials divisible by $r$ and with signature pair $(A,B)$.
Put $M=2(2n^2-n)$. Then the following hold:
\begin{itemize}

\item ${\mathcal H}(n;0,0) \cap I(r) = \{0\}$.

\item If $A=0$ or $B=0$ but not both are $0$, then ${\mathcal H}(n;A,B) \cap I(r)$ is empty.

\item ${\mathcal H}(n;A,1) \cap I(r)$ is empty for $A<n$ and ${\mathcal H}(n;1,B) \cap I(r)$ is empty for $B<n$.

\item If $A=1$ and $B \ge M-1$, or if $B=1$ and $A \ge M-1$, then ${\mathcal H}(n;A,B) \cap I(r)$ is non-empty, but 
${\mathcal D}^*\left ({\mathcal H}(n;A,B) \cap I(r) \right)$ is finite.

\item If $A+B \ge M$, and $A,B \ge 2$, then  ${\mathcal H}(n;A,B) \cap I(r)$ is non-empty and ${\mathcal D}^*\left({\mathcal 
H}(n;A,B) \cap I(r)\right)$ is infinite. 

\end{itemize}\end{theorem}

\begin{proof} Certainly ${\mathcal H}(n;0,0) \cap I(r)$ is non-empty, as $0$ is a multiple of $r$.
Next we list some pairs $(A,B)$ for which $J={\mathcal H}(n;A,B) \cap I(r)$ is empty. First suppose that $B=0$.
Then we have a bihomogeneous polynomial $p= ||f||^2$ which is divisible by $||z||^2 - |z_{n+1}|^2$. This situation cannot
occur unless $f=0$, because the zero set of $||f||^2$ must be a complex algebraic variety, and the only such variety containing
the zero set of $p$ is the whole space. Therefore ${\mathcal H}(n;A,0) \cap I(r)$ is empty for $A>0$.
By the same reasoning applied to $-p$, we see that ${\mathcal H}(n;0,B) \cap I(r)$ is empty for $B>0$.
These remarks prove the first two items.

Suppose next that $B=1$. Then we have a bihomogeneous
polynomial $q= ||f||^2 - |g|^2$
 divisible by $p$ with ${\bf s}(q)=(A,1)$.
 The map ${f \over g}$ is then a rational
 holomorphic map from the
 unit sphere to the unit sphere in ${\bf C}^A$. This map
 cannot be constant, because then $q$ would be zero,
 and $A=B=0$. Therefore ${f \over g}$
is a proper holomorphic rational mapping between
 the unit ball in the domain and the unit ball in the target.
 By Proposition 5.1
the target dimension must be at least as large as the domain dimension. 
Therefore, if $A <n$, then ${\mathcal H}(n;A,1) \cap I(r)$ is also empty. As before (using $-q$)
it follows that if $B<n$, then ${\mathcal H}(n;1,B) \cap I(r)$ is also empty. 

On the other hand ${\mathcal H}(n;n,1) \cap I(r)$ is non-empty, as it contains the identity map. Because of the gap phenomenon [HJX] we cannot conclude
that ${\mathcal H}(n;A,1) \cap I(r)$ is non-empty
 for all $A$ such that $A \ge n$. By Theorem 8.1 however
there is an integer $N_0$ with the following property: for each $N$ with $N \ge N_0$, there is a proper monomial mapping $f$ from the unit ball $B_n$
to $B_N$ for which $N$ is the minimal imbedding dimension. Furthermore, 
the number $N_0$ (given explicitly in Theorem 8.1)
is quadratic in $n$, as a consequence of the postage stamp problem. After homogenizing we obtain 
a bihomogeneous polynomial in ${\mathcal H}(n;N,1) \cap I(r)$ that is not in  ${\mathcal H}(n;N-1,1) \cap I(r)$. 
As before, after multiplying by $-1$, we see that ${\mathcal H}(n;1,N) \cap I(r)$ is non-empty
for $N \ge N_0$. We next show that 
${\mathcal D}^*({\mathcal H}(n;1,N) \cap I(r))$ is finite. To do so, it suffices to show that there is some bound $c(n,N)$ on the degree
of a proper rational mapping between the unit ball $B_n$ and the unit ball $B_N$. In fact the inequality (for $n\ge 2$)
$$ d \le {N(N-1) \over 2(2n-3)}$$
for the degree was proved in [DL]. 

Next we must show, for each  $A,B \ge 2$ and $A+B$ sufficiently large,
 that there is a  bihomogeneous polynomial
$q$, divisible by $||z||^2 - |z_{n+1}|^2$,  with ${\bf s}(q) = (A,B)$.
 We first assume that $N\ge N_0$, where 
$N_0$
is as in the previous paragraph. We therefore have
 a bihomogeneous polynomial map, say $p = ||f||^2 - |g|^2$,  of degree 
$(d,d)$
with signature pair $(N,1)$. We may assume that it is not
 divisible by $|z_{n+1}|^2$. Choose an arbitrary large integer $k$.

We write $p=vr$. Consider as before $q= (Hv \pm |z_1|^{2k}) r$. As in the case $n=2$ we conclude that
${\bf s}(q) = (N,1) + (n,1)$ if the plus sign is chosen, and ${\bf s}(q) = (N,1) + (1,n)$ if the minus sign is chosen.
As before it follows that we can find elements of ${\mathcal H}(n;A,B) \cap I(r)$ with arbitrarily large degree whenever
there are nonnegative $a$ and $b$, at least one of them positive,
with one of the following:
$$ (A,B) = (N,1) + a(n,1) + b(1,n),$$
$$ (A,B) = (1,N) + a(n,1) + b(1,n).$$
We finish the proof by showing that the set of such lattice points 
(as $N$, $a$, and $b$ vary) includes all pairs with $A+B$ sufficiently large and both at least $2$.

We claim that this set includes all $(A,B)$ such that
$A+B \ge 2(n^2-2n)$.  To see why, we first assume without loss of generality that $A \ge B$.  
Then $(A,B) = (N,1)+a(n,1)+b(1,n)$ gives $ N = A - an -b$.
Put $a= B-bn-1$. By Theorem 8.1 or Corollary 8.1, we need to make $N \ge 
n^2-2n+2$.
We get
$$ A-Bn+bn^2 + n - b = N \ge  n^2-2n+2, $$
which simplifies to
$$ A - Bn + b(n^2-1) \ge n^2-3n+2. \eqno (23) $$
Since $A \ge B$, to prove (23) it suffices to show that
$$ B-Bn+ b(n^2-1) \ge n^2-3n+2, $$
which is equivalent to
$$ b \ge {n-2 \over n+1} + {B \over n+1}. \eqno (24) $$
Put $b = \lfloor \frac{B}{n} \rfloor$ in (24). Then (24) is satisfied if

$$ \lfloor \frac{B}{n} \rfloor \ge {B \over n} - 1 \ge {n-2 \over n+1} + {B \over n+1}. \eqno (25) $$
But the right-hand inequality in (25) holds if and only if $B \ge 2n^2-n$.
Thus, if $A+B \ge 2B \ge M$, then $N \ge n^2-2n+2$. Therefore
${\mathcal H}(n;N,1) \cap I(r)$ is nonempty.
\end{proof}

\begin{corollary} Given the pair $(A,B)$ with $A+B\ge M$ and $A,B \ge 2$,
there exist rational mappings from  $S^{2n-1}$ to a hyperquadric $Q(A,B-1)$ whose images lie in no complex hyperplane
and whose degrees can be chosen arbitrarily large. 
\end{corollary}

It is possible to state the Corollary in several equivalent fashions.
For example, in the projective setting, where the mappings are written in homogeneous coordinates and the target
is $HQ(A,B)$, the condition on the image is equivalent to assuming that
the components of the mapping are linearly independent.

\section{bibliography}

[BEH] Baouendi, M. S., Ebenfelt, P., and Huang, X., Holomorphic mappings between hyperquadrics
with small signature difference, arXiv:0906.1235.
\medskip

[BH] Baouendi, M. S.and Huang, X., Super-rigidity for holomorphic mappings between hyperquadrics with positive signature,
J. Differential Geom.  69  (2005),  no. 2, 379-398. 
\medskip

[CD] Catlin, D. and D'Angelo, J., A stabilization theorem for 
Hermitian forms and applications to holomorphic mappings, {\it Math. Res. Lett.} 3 (1996), 149-166.
\medskip

[D1] D'Angelo, J., Several Complex Variables and the Geometry of Real Hypersurfaces,
CRC Press, Boca Raton, 1993.
\medskip

\medskip
[D2] D'Angelo, J., Proper holomorphic mappings, 
positivity conditions, and isometric imbedding, {\it J. Korean Math Society}, May 2003, 1-30.

\medskip
[D3] D'Angelo, J., Inequalities from Complex Analysis, Carus Mathematical Monograph No. 28, 
Mathematics Association of America, 2002.

\medskip

[D4] D'Angelo, J., Invariant CR Maps, in 
Complex Analysis: {\it Several complex variables and connections with PDEs and geometry} (Fribourg 2008),
Trends in Mathematics, Birkhauser, to appear.

\medskip
[DL] D'Angelo, J. and Lebl, J.,
Complexity results for CR mappings between spheres,
Int. J. of Math., Vol. 20, No. 2 (2009),  149-166.

\medskip
[DL2] D'Angelo, J. and Lebl, J.,
 On the complexity of proper holomorphic mappings between balls.  Complex Var. Elliptic Equ.  54  (2009),  no. 3-4, 187-204.

\medskip

[DKR] D'Angelo, J., Kos, \v S., and Riehl, E.,
A Sharp Bound for the Degree of Proper Monomial
Mappings Between Balls, {\it J. Geometric Analysis}, Volume 13 (2003), no. 4,
581-593. 

\medskip 

[DLP] D'Angelo, J., Lebl, J.,  and Peters, H.,
Degree estimates for polynomials constant on hyperplanes,  {\it Michigan Math. J.}, 55 (2007), no. 3, 693-713.
\medskip

\medskip
[EHZ] Ebenfelt, P., Huang, X.,  and Zaitsev, D., Rigidity of CR-immersions into spheres,
{\it Comm. Anal. Geom.},  12  (2004),  no. 3, 631-670.

\medskip
[Fa1] Faran, J., On the linearity of 
proper maps between balls in the low codimension case,
{\it J. Diff. Geom.}, 24 (1986), 15-17.

\medskip
[Fa2] Faran, J., Maps from the two-ball to the three-ball, 
{\it Inventiones Math.}, 68 (1982), 441-475.

\medskip
[F] Forstneric, F., 
Extending proper holomorphic maps of positive codimension, 
{\it Inventiones Math.}, 95(1989), 31-62.
\medskip

\medskip
[H] Huang, X., On a linearity problem for proper maps between balls in complex spaces
of different dimensions, {\it J. Diff. Geometry} 51 (1999), no 1, 13-33.

\medskip
[HJ] Huang, X., and Ji, S., 
Mapping $B_n$ into $B_{2n-1}$, {\it Invent. Math.} 145 (2001), 219-250.
13-36.
\medskip

[HJX] Huang, X., Ji, S.,  and D. Xu, 
A new gap phenomenon for proper holomorphic mappings from $B\sp n$ into $B\sp N$.
{\it Math. Res. Lett.} 13 (2006), no. 4, 515--529. 

\medskip
[L] Lebl, J., Normal forms, Hermitian operators, and CR maps of spheres and hyperquadrics,  arXiv:0906.0325.

\medskip
[LP] Lebl, J. and Peters, H., Polynomials constant on a hyperplane and CR maps of hyperquadrics, arXiv:0910.2673.

\medskip

[Mq] Marques de Sá, E.,
On the inertia of sums of Hermitian matrices,
Linear Algebra Appl. 37 (1981), 143--159.
\medskip

[M] Meylan, F., Degree of a holomorphic map between unit balls from ${\bf C}^2$ 
to ${\bf C}^n$,  {\it Proc. Amer. Math. Soc.}  134  (2006),  no. 4, 1023-1030.

\end{document}